\documentclass[11pt]{amsart}
\usepackage{amsmath,amsfonts,amsthm,mathrsfs,amssymb,amscd,comment,enumerate,amsxtra,url,tikz-cd}
\input{xypic}
\xyoption{all}

\usepackage{xcolor} 
\colorlet{mdtRed}{red!50!black}
\definecolor{dblue}{rgb}{0,0,.6}
\usepackage[colorlinks]{hyperref}
\hypersetup{linkcolor=blue,citecolor=dblue,filecolor=dullmagenta,urlcolor=mdtRed}

\numberwithin{equation}{section}
\newtheorem{theorem}[equation]{Theorem}
\newtheorem{corollary}[equation]{Corollary}
\newtheorem{lemma}[equation]{Lemma}
\newtheorem{proposition}[equation]{Proposition}
\newtheorem{definition}[equation]{Definition}

\newtheorem*{theorem*}{Theorem}
\newtheorem*{corollary*}{Corollary}

\theoremstyle{remark}
\newtheorem{remark}[equation]{Remark}

\DeclareMathOperator{\et}{\textnormal{\'et}}

\newcommand{\mf}[1]{\mathfrak{#1}}

\newcommand{\mb}[1]{\mathbb{#1}}
\newcommand{\mc}[1]{\mathcal{#1}}

\begin{document}

\title[On Fundamental Group-Scheme]{Fundamental Group Schemes of some Quot schemes on a Smooth Projective Curve} 

\author[C. Gangopadhyay]{Chandranandan Gangopadhyay} 

\address{Department of Mathematics, Indian Institute of Technology Bombay, Powai, Mumbai 400076, Maharashtra, India.} 

\email{chandra@math.iitb.ac.in} 

\author[R. Sebastian]{Ronnie Sebastian} 

\address{Department of Mathematics, Indian Institute of Technology Bombay, Powai, Mumbai 400076, Maharashtra, India.} 

\email{ronnie@math.iitb.ac.in} 



\date{\today}

\begin{abstract}
	Let $k$ be an algebraically closed field. 
	Let $C$ be an irreducible smooth projective curve over $k$. 
	Let $E$ be a locally free sheaf on $C$ of rank $\geq 2$.
	Fix an integer $d \geq 2$. Let $\mc Q$ denote the Quot scheme
	parameterizing torsion quotients of $E$ of degree $d$. 
	In this article we compute the $S$-fundamental group scheme of $\mc Q$. 
\end{abstract}

\maketitle

\section{Introduction}
Let $X$ be a connected, reduced and complete scheme over a perfect field $k$ 
and let $x \in X$ be a $k$-rational point.
In \cite{No1}, Nori introduced a $k$-group scheme $\pi^N(X,x)$ associated to essentially 
finite vector bundles on $X$. In \cite{No2}, Nori extends the definition of $\pi^N(X,x)$ 
to connected and reduced $k$-schemes. 
In \cite{BPS}, Biswas, Parameswaran and Subramanian defined the notion of {\it $S$-fundamental group scheme} 
$\pi^S(X, x)$ for $X$ a smooth projective curve over any algebraically closed field $k$. 
This is generalized to higher dimensional connected smooth projective $k$-schemes 
and studied extensively by Langer in \cite{La, La2}.

Let $C$ be a connected smooth projective curve defined 
over an algebraically closed field $k$. 
Fix a locally free sheaf $E$ on $C$ of rank $\geq 2$ and
an integer $d \geq 2$. Let $\mc Q$ denote the Quot scheme
parameterizing torsion quotients of $E$ of degree $d$. 
It is a smooth and projective variety over $k$.
In this article we shall compute the $S$-fundamental group scheme
of $\mc Q$. We mention some of the earlier results where fundamental
group schemes were computed. In \cite{BH} it is proved that for a smooth projective 
surface $X$, the etale fundamental group $\pi^{\et}(Hilb_X^n,nx)$ of the Hilbert 
scheme of $n$ points ($n\geq 2$), is isomorphic to $\pi^{\et}(X,x)_{\rm ab}$. The main 
result in \cite{PS-surface} is to generalize this to the $S$-fundamental
group scheme. In \cite{La2} it is proved that $\pi^S({\rm Alb}(C),0)$ 
is isomorphic to $\pi^S(C,c)_{\rm ab}$. 
Let $S_d$ ($d\geq 2$) be the permutation group of $d$ symbols
and denote $S^dC := C^d/S_d$.
In \cite{PS-curve} the authors 
prove that the $\pi^S(S^dC,d[c])$ is isomorphic to $\pi^S(C,c)_{\rm ab}$.
Once we have such a result for the $S$-fundamental group scheme, one deduces similar results 
for the Nori and etale fundamental group schemes.

\noindent 
{\bf Notation}. From now on, unless mentioned otherwise, we will be working with 
the following assumptions. Let $k$ be an algebraically closed field. 
Let $C$ be an irreducible smooth projective curve over $k$. 
Let $E$ be a locally free sheaf on $C$ of rank $\geq 2$.
Fix an integer $d \geq 2$. Let $\mc Q$ denote the Quot scheme
parameterizing torsion quotients of $E$ of degree $d$. 
There is a Hilbert-Chow map $\phi:\mc Q\to S^dC$, the definition of which is 
recalled in section \ref{section-Hilbert-Chow}.

The $S$-fundamental group scheme has been defined in Definition \ref{cat-nf}.
The main result we prove in this article is the following.
\begin{theorem*}[Theorem \ref{m-t}]
	For any closed point $q \in \mc Q$, there is an isomorphism of affine $k$-group schemes 
	$${\phi_*^S} : \pi^S(\mc Q, q) \xrightarrow{\sim} \pi^S(S^dC, \phi(q)).$$
\end{theorem*}

\noindent
From this the following corollary follows easily. 

\begin{corollary*}[Corollary \ref{m-c}]
For any closed point $q \in \mc Q$, there are isomorphism of affine $k$-group schemes 
${\phi_*^N} : \pi^N(\mc Q, q) \xrightarrow{\sim} \pi^N(S^dC, \phi(q))$ and 
${\phi_*^{\et}} : \pi^{\et}(\mc Q, q) \xrightarrow{\sim} \pi^{\et}(S^dC, \phi(q)).$
\end{corollary*}

\noindent
In view of \cite{PS-curve} it follows that 
\begin{corollary}$\pi^{?}(\mc Q,q)\cong \pi^{?}(C,c)_{\rm ab}$
for $?=S,N,\et$.
\end{corollary}

The key ingredient in the proof of the above theorem is the following corollary 
regarding the scheme theoretic fiber of the Hilbert-Chow morphism. Related to this,
in \cite[Proposition 5.9]{Kleiman-et-al}, the authors give a constructive 
proof that each fiber of the
Hilbert-Chow map has the same reduction as the product of certain Quot
schemes, which, a priori, need not be reduced. We prove the following. Let $D$ 
be a divisor corresponding to a closed point of $S^dC$. Let $\mc Q_D$ denote the 
scheme theoretic fiber over the point $D$. Then we have the following result. 
\begin{corollary*}[Corollary \ref{cor-fiber}]
The fiber $\mc Q_D$ is reduced, irreducible and normal.
\end{corollary*}
\noindent
Further, there is a smooth projective rational variety $S_d$ and a birational 
map $g_d:S_d\to \mc Q_D$, see Proposition \ref{birational}. 
This allows us to conclude easily, using Grauert's theorem, 
that every numerically flat bundle is the pullback of a numerically flat along $\phi$.

In the rest of this article, $E$ will be a locally free sheaf of 
rank $\geq 2$ and $d$ will be an integer $\geq 2$.

\subsection*{Acknowledgements} We thank Arjun Paul for useful discussions.
We thank the authors in \cite{Kleiman-et-al} for their interest. We thank
the referee for a very careful reading of this article and for 
helpful comments.

\section{Hilbert-Chow morphism}\label{section-Hilbert-Chow}
In this section we recall the definition of the Hilbert Chow morphism 
$\phi:\mc Q\to S^dC$. For Hilbert schemes of points, one has a 
Hilbert-Chow morphism, see \cite[Chapter 7, section 1]{FGAex}
for a detailed discussion. Here we describe how to get such 
a morphism for the Quot schemes we consider. This construction appears in 
other places, for example, see the introduction in 
\cite{Bis-Dh-Hu}. The map $\phi$ that we define here is the same 
as the map $\xi$ defined in \cite{Kleiman-et-al}, in the situation
that they work in. There, however, the authors give a more explicit
description.

Let $p_1:C\times \mc Q\to C$
and $p_2:C\times \mc Q\to \mc Q$ denote the projections
and let 
\begin{equation}\label{hc-e1}
0\to K\to p_1^*E\to B\to 0
\end{equation}
denote the universal quotient on $C\times \mc Q$.
Since $B$ and $p_1^*E$ are flat over $\mc Q$, it follows 
that $K$ is a flat $\mc Q$ sheaf. Let $q\in \mc Q$ denote 
a closed point. Restricting this quotient to 
$C\times q$ we get the exact sequence
\begin{equation}\label{hc-e2}
0\to K\vert_{C\times q}\to E\to B\vert_{C\times q}\to 0\,. 
\end{equation}
It follows that $K\vert_{C\times q}$ is a locally free 
sheaf on $C$. From Nakayama's lemma it follows that 
$K$ is a locally free $C\times \mc Q$ sheaf of rank $r:={\rm rank}\, E$.
Taking determinant of the inclusion in \eqref{hc-e1} we get an exact sequence
$$0\to {\rm det}(K)\to {\rm det}(p_1^*E)\to \mc F\to 0\,.$$
To show that $\mc F$ is flat over $\mc Q$ it suffice 
to show that the restriction of this sequence to 
$C\times q$ remains exact on the left. But this is clear
as the restriction of this sequence to $C\times q$ 
is precisely the sequence obtained by taking determinant
of the inclusion in \eqref{hc-e2}, which remains exact on the left.
Thus, on $C\times \mc Q$ we get a quotient 
$$0\to {\rm det}(K)\otimes{\rm det}(p_1^*E)^{-1}\to \mc O\to \mc F\otimes {\rm det}(p_1^*E)^{-1}\to 0\,.$$
This defines a morphism
\begin{equation}\label{defn-hc}
\phi:\mc Q\to S^dC\,.
\end{equation}
In the following sections we will study the fibers of this morphism.

\section{Locus where $\phi$ is smooth}
Consider the map $\phi:\mathcal{Q}\rightarrow S^{d}C$.
Let $D$ denote the divisor $\sum^{k}_{i=1}d_{i}[c_{i}]$
and consider a quotient $q$ 
\begin{equation}\label{thick-points}
E\xrightarrow{q} \mathcal{O}_{D}\rightarrow 0\,. 
\end{equation}

\begin{lemma}\label{GS-L2}
Given a quotient $q$ as above, there is a line bundle $L$ 
and a surjection $E\rightarrow L\rightarrow 0$
such that $q$ factors as 
$$E\rightarrow L\rightarrow \mc O_D \,.$$
\end{lemma}
\begin{proof}
Let $L'$ be any line bundle on $C$.
Then we have the exact sequence
\[
0 \to L'(-D) \to L' \to L'|_{D} \to 0\,.
\]
Applying the functor $\text{Hom}(E, )$ to the above exact sequence, we get
\[
0 \to \text{Hom}(E,L'(-D)) \to \text{Hom}(E,L') \to \text{Hom}(E,L'|_{D}) \to \text{Ext}^{1}(E,L'(-D))
\]
Now $\text{Ext}^{1}(E,L'(-D))\cong H^{1}(E^{\vee}\bigotimes L'(-D))$, 
so for $L'$ of sufficiently high degree we get $\text{Ext}^{1}(E,L'(-D))=0$,
that is, we have an exact sequence when deg $L'\gg0$
\[
0 \to \text{Hom}(E,L'(-D)) \to \text{Hom}(E,L') \to \text{Hom}(E,L'|_{D}) \to 0\,.
\]
In other words, for any homomorphism $E\rightarrow L'|_{D}$, 
we have a morphism $E\rightarrow L'$ 
such that the following diagram commutes:
\[
\begin{tikzcd}
E \arrow[r] \arrow[d] & L'|_{D}\cong \mc O_D \\
L '\arrow[ur] &
\end{tikzcd}
\]
In particular, taking the quotient $q:E\to \mc O_D$, 
there is a line bundle $L'$ such that $q$ factors 
as $E\to L'\to \mc O_D$. Let $L$ be the image of $E$ in $L'$. 
Then, we have a surjection 
$$E\rightarrow L\rightarrow \mc O_{D}\,,$$
which proves the lemma.
\end{proof}
\begin{lemma}\label{GS-L3}
The map $\phi:\mathcal{Q}\rightarrow S^{d}C$ 
is smooth at $q$ (corresponding to the quotient in equation 
\eqref{thick-points}).
\end{lemma}
\begin{proof}
We will show that the map of Zariski tangent spaces
$T_{q}\mathcal{Q}\rightarrow T_{\phi(q)}S^{d}C$
is surjective. Let $T:={\rm Spec}\,k[\epsilon]/(\epsilon^2)$
and let $t_0$ denote the closed point of $T$.
We will show that for any map
$$T\xrightarrow{v} S^{d}C$$ 
such that the image of the closed point of $T$ is 
$\phi(q)$, there is a map 
$T\xrightarrow{v'} \mathcal{Q}$, 
such that the closed point maps to $q$ and the 
following diagram commutes
\[
\begin{tikzcd}
& \mathcal{Q} \arrow[d,"\phi"] \\
T \arrow[r,"v"] \arrow[ru,dashed,"v'"] & S^{d}C
\end{tikzcd}
\]
By universal property of $S^{d}C$, 
the morphism $v$ corresponds to a quotient over $C\times T$ given by
$$\mathcal{O}_{C\times T}\xrightarrow{J_v} \mathcal{O}_\mathcal{D}\rightarrow 0$$
such that $\mathcal{O}_{\mathcal{D}}$ is $T$-flat and the restriction of $J_v$ to 
$C\times {\rm Spec}\,k$ is equivalent to the quotient 
$[\mathcal{O}_C\rightarrow \mathcal{O}_{D}]$ (which corresponds to the point $\phi(q)$). 
Note that $\mc O_{\mc D}$ is an Artinian ring.

Let $f_1:C\times T\to C$ and $f_2:C\times T\to T$ denote the projections.
We fix a line bundle $L$ over $C$ as in Lemma \ref{GS-L2}. 
We have $1\otimes J_v:f_1^*L\to f_1^*L\otimes \mc O_{\mc D}$.
Define a quotient $J_{v'}$ over $C\times T$ as the composition
\begin{equation}\label{lift-tv}
J_{v'}:f^{*}_{1}E\rightarrow f^{*}_{1}L\xrightarrow{1\otimes J_v} f^{*}_{1}L\otimes \mc O_{\mc D}\cong \mc O_{\mc D}\,. 
\end{equation}

Clearly, $f^{*}_{1}L\otimes \mc O_{\mc D}$ is $T$-flat and 
$J_{v'}|_{C\times \{\text{Spec }k\}}$ is equivalent to $q$ by Lemma \ref{GS-L2}.
Hence, $J_{v'}$ induces a morphism $v':T\rightarrow \mathcal{Q}$. 
Next we show that $\phi\circ v'=v$. Let us denote the kernel of $J_{v'}$ 
by $E_{v'}$. Thus, we have an exact sequence
$$0\to E_{v'}\stackrel{\iota}{\longrightarrow} f_1^*E\stackrel{J_{v'}}{\longrightarrow} \mc O_{\mc D}\to 0\,.$$
Let $t_0$ denote the closed point of $T$. 
From the $T$-flatness of $\mc O_{\mc D}$ we conclude that 
$E_{v'}\vert_{C\times t_0}$ is locally free. Now using Nakayama's 
lemma we conclude that $E_{v'}$ is locally free sheaf on $C\times T$.
Recall from the definition of $\phi$, that the map 
$\phi\circ v'$ is given by the following quotient on $C\times T$
$$0\to {\rm det}(E_{v'})\otimes {\rm det}(f_1^*E)^{-1}\xrightarrow{{\rm det}(\iota)} \mc O_{C\times T}\to \mc F\to 0\,.$$
Passing to the local
rings at $(c,t_0)$ and using equation \eqref{lift-tv}, 
it is easily checked that $\mc F\cong \mc O_{\mc D}$
and that the above quotient is exactly $J_v:\mc O_{C\times T}\to \mc O_{\mc D}$.
This completes the proof of the lemma.
\end{proof}

\begin{proposition}\label{Generic Smoothness}
In every fiber of $\phi$ there is a point at which 
$\phi$ is a smooth morphism.
\end{proposition}
\begin{proof}
Let $D$ be the divisor corresponding to a point $x\in S^dC$.
Fix a line bundle $L$ which is a surjective quotient of 
$E$. Then the composite $q:E\to L\to L\otimes \mc O_D$ 
is a quotient such that $\phi(q)=x$. The proposition now 
follows from lemma \ref{GS-L3}.
\end{proof}

\section{The space $S_d$}\label{S_d}
Let the rank of the vector bundle 
$E$ be $r$. We will inductively define  
$(S_{d},A_{d})$, where $A_{d}$ is a vector bundle of rank $r$ defined over 
$C\times S_{d}$. 

Define 
$S_{0}={\rm Spec}\,k$ and $A_{0}=E$. To define $(S_{j},A_{j})$ 
we assume that we have defined $(S_{j-1},A_{j-1})$. 
Let $i_{j-1}:\{c\}\times S_{j-1}\hookrightarrow C\times S_{j-1}$ 
be the natural closed immersion.
Define 
$S_{j}:=\mathbb{P}(i^{*}_{j-1}A_{j-1})$ 
and let $f_{j,j-1}:S_{j}\rightarrow S_{j-1}$ be the structure morphism. 
Finally let $F_{j,j-1}:=id_C\times f_{j,j-1}:C\times S_{j}\rightarrow C\times S_{j-1}$. 
Let $p_{1,j}$ and $p_{2,j}$ be the projections from $C\times S_{j}$ 
to $C$ and $S_{j}$, respectively. 

For each $j$, we have the following diagram
\[
\begin{tikzcd}
\{c\}\times S_{j} \arrow[r,hook,"i_j"] \arrow[rd, "="]& C\times S_{j} \arrow[r,"F_{j,j-1}"] 
		\arrow[d,"p_{2,j}"] & C\times S_{j-1} \arrow[d,"p_{2,j-1}"] \\
	& S_j \arrow[r,"f_{j,j-1}"] & S_{j-1} \arrow[u, bend right=90, "i_{j-1}", labels=below right]
\end{tikzcd}
\]
Let $\mathcal{O}_{j}(1)$ the universal line bundle over $S_{j}$.
Then over $C\times S_{j}$ we have the quotient
\begin{align*}
F^{*}_{j,j-1}A_{j-1}\rightarrow & (i_{j})_{*}i^{*}_{j}F^{*}_{j,j-1}A_{j-1} \\
                              = & (i_{j})_{*}f^{*}_{j,j-1}i_{j-1}^*A_{j-1} \\
                    \rightarrow & (i_{j})_{*}\mathcal{O}_{j}(1)
\end{align*}
Define $A_{j}$ to be the kernel of the above quotient. Since 
$(i_j)_*\mc O_j(1)$ is flat over $S_j$, restricting the exact sequence 
\begin{equation}\label{univ-seq-A_j}
	0\to A_j\to F^*_{j,j-1}A_{j-1}\to (i_j)_*\mc O_j(1)\to 0 
\end{equation}
to $C\times s$ we see that $A_j\vert_{C\times s}$ is torsion free
and so is locally free. It follows from Nakayama's lemma 
that $A_j$ is locally free on $C\times S_j$. 
Thus, we have defined $(S_j,A_j)$. It is easy to see, for example, 
using equation \eqref{e10} in the proof of the next Lemma, that closed 
points of $S_d$ are in 1-1 correspondence with filtrations
\begin{equation}\label{e8}
	E_d\subset E_{d-1} \subset E_{d-2} \subset \cdots \subset E_0=E
\end{equation}
where each $E_j$ is a locally free sheaf of rank $r$ on $C$
and $E_j/E_{j+1}$ is a skyscraper sheaf of rank one supported at 
$c\in C$. 

\section{Birationality of $S_d$ and $\mc Q_{d[c]}$}\label{subsection-birational}
Define the following morphisms for $j>i$:
\begin{align*}
	f_{j,i}&=f_{j,j-1}\circ \ldots \circ f_{i+1,i}:S_{j}\to S_{i},\\
	F_{j,i}&=F_{j,j-1}\circ \ldots \circ F_{i+1,i}:C\times S_{j}\to C\times S_{i}\,.
\end{align*}
Note that both of these morphisms are flat.
Let $V\subset \mc Q_{d[c]}$ be the open subset 
whose points parameterize quotients of the type $E\to \mc O_C/\mf m^d_{C,c}$. 

\begin{lemma}\label{fibre-l2}
There exists a morphism $g_{d}:S_{d}\rightarrow \mc Q_{d[c]}$ such that 
\begin{enumerate}[(i)]
	\item $g_d$ is surjective on closed points,
	\item $g_d^{-1}(V)\to V$ is a bijection,
	\item $S_d\setminus g_d^{-1}(V)\to \mc Q_{d[c]}\setminus V$ has positive dimensional fibers.
\end{enumerate}
\end{lemma}
\begin{proof}
We will define a quotient of $p_1^*E$ on $C\times S_{d}$.
Using the flatness of $F_{d,i}$, we have inclusions (recall the definition 
of $A_j$ from \eqref{univ-seq-A_j})
\begin{equation}\label{e10}
	A_{d}\subset F_{d,d-1}^{*}A_{d-1}\subset F_{d,d-2}^{*}A_{d-2} \subset \ldots 
		\subset F_{d,1}^{*}A_{1}\subseteq p^{*}_{1}E. 
\end{equation}
Define 
\begin{align*}
	B_{j}^d&:=p^{*}_{1}E/F_{d,j}^{*}A_{j}\\
		&\cong F_{d,j}^{*}(p^{*}_{1}E/A_j)
\end{align*}
For each $j$ there is an exact sequence on $C\times S_d$
\begin{equation}\label{eq1}
0 \to F_{d,j-1}^{*}A_{j-1}/F_{d,j}^{*}A_j \to  B_{j}^d \to B_{j-1}^d \to 0\,.
\end{equation}
On $C\times S_j$ we have the quotient \eqref{univ-seq-A_j} 
$$0\to A_j\to F_{j,j-1}^*A_{j-1}\to F_{j,j-1}^{*}A_{j-1}/A_j\cong (i_j)_*(\mc O_j(1))\to 0\,.$$
As $F_{j,j-1}^{*}A_{j-1}/A_j$ is $S_j$-flat, the pullback along $F_{d,j}$, that is,  
$F_{d,j-1}^{*}A_{j-1}/F_{d,j}^{*}A_j$
is $S_{d}$-flat. When restricted to $C\times s$ 
for $s\in S_{d}$, it is a degree one torsion sheaf supported 
at $c$. By induction on $j$, using equation \eqref{eq1}, 
one sees that $B_j^d$ is $S_d$-flat and 
the restriction of $B_j^d$ to $C\times s$ is a torsion sheaf 
of degree $j$ supported at $c$. In particular, 
\begin{equation}\label{e2}
0\to A_d\to p^{*}_{1}E\rightarrow B^{d}_{d} \to 0
\end{equation}
is  a quotient such that 
$B^{d}_{d}$ is $S_{d}$ flat and for each 
$s\in S_{d}$, $B^{d}_{d}|_{C\times \{s\}}$ is a torsion sheaf 
of degree $d$ supported at $c$.
By the universal property of $\mathcal{Q}$, we have a morphism 
$$S_{d}\rightarrow \mathcal{Q},$$ 
such that the set theoretic image of the composition 
$$S_{d}\rightarrow \mathcal{Q}\xrightarrow{\phi} S^{d}C$$ 
is the point $d[c]$. Since $S_{d}$ is reduced,
the scheme theoretic image is the scheme 
$\{d[c]\}\hookrightarrow S^{d}C$.
In other words, we get that the above morphism factors as
\[
\begin{tikzcd}
S_{d} \arrow[r,"g_{d}"] & \mc Q_{d[c]} \arrow[r,hook ] & \mathcal{Q}\,.
\end{tikzcd}
\]

A closed point of $S_d$ corresponds to a filtration as in 
\eqref{e8}. Under $g_d$ this point maps to the quotient $E\to E/E_d$.
Conversely, given a closed point of $\mc Q_{d[c]}$ it is clear that we 
can find a closed point of $S_d$ which maps to it. This proves (i). 
Suppose $E_d\subset E$ is such that $E/E_d\cong \mc O/\mf{m}^d_{C,c}$,
then for every $0\leq j\leq d$ there is a unique $E_j$ 
such that $E_d\subset E_j\subset E$ and $E/E_j\cong \mc O/\mf{m}^j_{C,c}$.
From this one easily concludes (ii).

For a closed point in $\mc Q_{d[c]} \setminus V$, corresponding to a quotient 
$E\to \mc F_d$, we have ${\rm rank}(\mc F_d\otimes \mc O/\mf{m}_{C,c})\geq 2$. 
We can construct infinitely many chains 
$\mc F_d\to \mc F_{d-1} \to \ldots \to \mc F_1$. 
Therefore, the fiber over such a closed point is positive 
dimensional. This proves $(iii)$.
\end{proof}

\begin{corollary}
The fiber $\mc Q_{d[c]}$ is irreducible of dimension $d(r-1)$. 
\end{corollary}
\begin{proof}
Since $S_d$ is irreducible, it is clear that $\mc Q_{d[c]}$ is irreducible.
It is clear that the dimension of $S_d$ is $d(r-1)$. Thus, 
the dimension of $\mc Q_{d[c]}$ is at most $d(r-1)$. On the other hand,
the dimension of the fiber of $\phi$ over a general point is $d(r-1)$. 
This shows that the dimension of $\mc Q_{d[c]}$ is at least $d(r-1)$.
\end{proof}

\begin{corollary}
The codimension of $\mc Q_{d[c]}\setminus V$ in $\mc Q_{d[c]}$ is 
$\geq 2$.
\end{corollary}
\begin{proof}
As $S_d$ and $\mc Q_{d[c]}$ have the same dimension and $S_d$ is irreducible, 
this follows easily using $(iii)$ in lemma \ref{fibre-l2}.
\end{proof}

\begin{corollary}\label{fiber-R1}
The fiber $\mc Q_{d[c]}$ satisfies Serre's condition $(R_1)$.
\end{corollary}
\begin{proof}
Since the map $\phi$ is smooth at points $v\in V$, it follows 
that $\mc O_{\mc Q_{d[c]},v}$ is a regular local ring for all $v\in V$. 
Further, from the preceding corollary $V$ contains all prime ideals
of height 1. The corollary follows.
\end{proof}

Next we will show that $g_d$ is birational. Let 
\begin{equation*}
p^{*}_{1}E\rightarrow B' 
\end{equation*}
be the restriction of the universal quotient $B$ 
over $C\times \mc Q$ to the subscheme $C\times\mc Q_{d[c]}$.
Let us define the inclusion 
$$i:{\rm Spec}\,(\mathcal{O}_{C,c}/\mf{m}^{d}_{C,c})\times \mc Q_{d[c]}\hookrightarrow C\times \mc Q_{d[c]}\,.$$ 

\begin{lemma}\label{fibre-l3}
There is a coherent sheaf $F_d$ over 
${\rm Spec}\,(\mathcal{O}_{C,c}/\mf{m}^{d}_{C,c})\times \mc Q_{d[c]}$
such that $B'=i_*F_d$.
\end{lemma}
\begin{proof}
It is enough to show that the $p^{*}_{1}(E\otimes \mathcal{O}(-dc))$ 
is contained in the kernel of $p_1^*E\to B'$. Denote the kernel by 
$A'$. Let $0\to E'\xrightarrow{h} E$ be locally free sheaves of the same rank on 
a scheme $Y$. Let $\mc I$ denote the ideal sheaf determined by 
${\rm det}(h)$. Then it is easy to see that $\mc IE\subset h(E')\subset E$.
Thus, it suffices to find the ideal sheaf corresponding to the following
exact sequence
\begin{equation}\label{univ-exact-Q_p}
0 \to A' \xrightarrow{h} p^{*}_{1}E \to B' \to 0
\end{equation}
on $C\times \mc Q_{d[c]}$.
By the definition of $\phi$, the map $\mc Q_{d[c]}\to \mc Q\xrightarrow{\phi} S^{d}C$
is given by the quotient 
$$0\to {\rm det}(A')\xrightarrow{{\rm det}(h)} {\rm det}(p_1^*E) \to \mc F\to 0$$
on $C\times \mc Q_{d[c]}$. But since the image of $\mc Q_{d[c]}$ under 
this morphism is precisely $d[c]$, it follows that this quotient is 
isomorphic to the quotient 
$$p_1^*\mc O_C\to p_1^*(\mc O_C/\mf m^d_{C,c})\,.$$
It is clear that the ideal sheaf $\mc I$ corresponding to 
the exact sequence \eqref{univ-exact-Q_p} is $p_1^*(\mc O_C(-dc))$.
The lemma now follows. 
\end{proof}

Define subschemes 
$$D_j:={\rm Spec}\,(\mathcal{O}_{C,c}/\mf{m}^{j}_{C,c})\times V\stackrel{\alpha_j}{\hookrightarrow} C\times V\,.$$
By the previous lemma there is a sheaf $F_d$ on 
$D_d$ such that $B'|_{C\times V}\cong (\alpha_d)_*(F_d)$. 
Clearly $F_d$ is flat over $V$ since $B'$ is.
\begin{lemma}\label{fibre-l4}
$F_d$ is a line bundle over $D_{d}$.
\end{lemma}
\begin{proof}
By definition of $V$, for each $v\in V$, 
$(F_d)_v\cong \mc O_C/\mf m^d_{C,c}$. Using $F_d$ is $V$-flat
and Nakayama's lemma we see that $F_d$ is a line bundle. 
\end{proof}

\begin{corollary}
The restriction $F_j:=F_d\vert_{D_j}$ is a line bundle on $D_j$.
\end{corollary}

\begin{remark}\label{comm-alg-rem}
We will use the following fact in the proof 
of the next theorem. Let $A$ and $B$ be rings and 
let $M$ be an $A\otimes_kB$ module. Let $B\to C$ be a 
ring homomorphism. Then 
$$M\otimes_{A\otimes_kB}(A\otimes_kC)\cong M\otimes_{A\otimes_kB}(A\otimes_kB)\otimes_BC\cong M\otimes_BC\,.$$ 
In particular, if $0\to N'\to N\to M\to 0$ 
is a short exact sequence of $A\otimes_kB$ modules,
where $M$ is flat as a $B$-module, then it remains exact 
when we apply the functor $-\otimes_{A\otimes_kB}(A\otimes_kC)$.
\end{remark}

\begin{proposition}\label{birational}
The restriction $g_d:g_d^{-1}(V)\to V$ is an isomorphism.
\end{proposition}
\begin{proof}
We will use induction on $j$ to define maps $V\to S_j$. 
Define $A'_0$ on $C\times V$ to be $p_1^*E$. 
For $j\geq 1$ define 
sheaves $A_j'$ on $C\times V$ as follows
\begin{equation}\label{e5}
0\to A_j'\to p_1^*E\to (\alpha_j)_*(F_j)\to 0\,. 
\end{equation}
Observe that we have a commutative diagram
\[
 \xymatrix{
	0\ar[r] & A_{j}'\ar[r]\ar[d] & p_1^*E \ar[r]\ar@{=}[d] & (\alpha_{j})_*(F_{j})\ar[r]\ar[d] & 0\\
	0\ar[r] & A_{j-1}'\ar[r] & p_1^*E \ar[r] & (\alpha_{j-1})_*(F_{j-1})\ar[r] & 0
 }
\]
The kernel of the right vertical arrow is $(\alpha_1)_*(F_1)$.
Thus, there is an exact sequence of sheaves on 
$C\times V$, 
\begin{equation}\label{exact-1-V}
0\to A_{j}'\to A_{j-1}'\xrightarrow{\delta_{j-1}} (\alpha_1)_*(F_1)\to 0\,. 
\end{equation}

Note that 
$S_1=\mb P(E_c)$. Thus, to give a map from $V\to S_1$ 
we need to give a line bundle quotient of $E_c\otimes O_V$.
Restricting the universal quotient $p_1^*E\to (\alpha_d)_*(F_d)$ 
on $C\times V$ to $c\times V$ we get the quotient
$E_c\otimes \mc O_V\to F_1\,.$
This defines a morphism $h_1:V\to S_1$.
On $C\times S_1$ one has the exact sequence 
\eqref{univ-seq-A_j}.
Using remark \ref{comm-alg-rem} we see that the pullback of this along 
$id_C\times h_1$ gives the following exact sequence
on $C\times V$,
\begin{equation*}
0\to (id_C\times h_1)^*A_1\to p_1^*E\xrightarrow{\delta_0} (\alpha_1)_*(F_1)\to 0\,. 
\end{equation*}
We see that $A_1'=(id_C\times h_1)^*A_1$.
Let us assume that we have constructed maps $h_{j-1}:V\to S_{j-1}$ such that 
the pullback of \eqref{univ-seq-A_j} along $id_C\times h_{j-1}$ yields the 
exact sequence 
\begin{equation}\label{e1}
0\to A_{j-1}'\to A_{j-2}'\xrightarrow{\delta_{j-2}} (\alpha_1)_*(F_1)\to 0\,. 
\end{equation}
Consider the diagram 
\[
\begin{tikzcd}
\{c\}\times V \arrow[r,hook,"\alpha_1"] \arrow[rd, "="]& C\times V \arrow[r,"id_C\times h_{j-1}"] 
		\arrow[d] & C\times S_{j-1} \arrow[d] \\
	& V \arrow[r,"h_{j-1}"] & S_{j-1} \arrow[u, bend right=90, "i_{j-1}", labels=below right]
\end{tikzcd}
\]
To give a map $V\to S_{j}$ we need to give a line bundle 
quotient of 
$$h_{j-1}^*i_{j-1}^*A_{j-1}\cong (\alpha_1)^*(id_C\times h_{j-1})^*A_{j-1}\cong (\alpha_1)^*A'_{j-1},$$
where the last isomorphism follows from \eqref{e1}.
Restricting \eqref{exact-1-V} to $c\times V$, 
we get a line bundle quotient 
$$(\alpha_1)^*A'_{j-1}\to F_1\,.$$
This defines a morphism $h_j:V\to S_j$. Pulling back 
\eqref{univ-seq-A_j} along $id_C\times h_j$, using \ref{comm-alg-rem}
and \eqref{e1}, 
we get the following short exact sequence on $C\times V$ 
$$0\to (id_C\times h_j)^*A_j\to A'_{j-1}\xrightarrow{\delta_{j-1}} (\alpha_1)_*(F_1)\to 0\,.$$
Now equation \eqref{exact-1-V} shows that $A'_j\cong (id_C\times h_j)^*A_j$
and there is an exact sequence
\begin{equation*}
0\to A_{j}'\to A_{j-1}'\xrightarrow{\delta_{j-1}} (\alpha_1)_*(F_1)\to 0\,. 
\end{equation*}
Thus, inductively we have constructed a map $h_d:V\to S_d$.

To show that the composite $V\xrightarrow{h_d}S_d\xrightarrow{g_d} \mc Q_{d[c]}$ 
is an isomorphism onto $V$, it suffices to show that the pullback 
of the universal quotient on $C\times \mc Q_{d[c]}$ along $id_C\times (g_d\circ h_d)$
is the restriction of the universal quotient to $C\times V$. 
Recall from \eqref{e5} the universal quotient on $C\times V$
is 
$$0\to A'_d\to p_1^*E\to B'\to 0\,.$$
Pulling this back along $id_C\times g_d$ is the quotient (recall from \eqref{e2})
$$0\to A_d\to p_1^*E\to B^d_d\to 0,$$
by the definition of the map $g_d$. From the definition of $h_d$, 
one easily checks that the pullback 
along $id_C\times h_d$ of the filtration in \eqref{e10} is the 
following filtration on $C\times V$,
\[
A'_d\subset A'_{d-1}\subset\cdots\subset p_1^*E\,.
\]
Thus, it follows that the pullback along $id_C\times h_d$ of 
\eqref{e2} is
$$0\to A'_d\to p_1^*E\to B'\to 0\, ,$$
which is the universal quotient on $C\times V$. This proves 
that $g_d\circ h_d$ is the identity on $V$.
By lemma \ref{GS-L3} the morphism $\phi$ is smooth 
at a point $v\in V$. This shows that the local ring 
$\mc O_{V,v}$ is a domain. Thus, we have maps 
$\mc O_{V,v}\to \mc O_{S_d,h_d(v)}\to \mc O_{V,v}$ 
such that the composite is the identity. Since both rings 
have the same dimension, the kernel of $\mc O_{S_d,h_d(v)}\to \mc O_{V,v}$
is forced to be 0, which shows that the local rings 
are isomorphic. This proves the proposition.
\end{proof}

\section{Normality of all fibers}
For a point $D=\sum d_i[c_i]\in S^dC$, denote by $\mc Q_{D}$ 
the scheme theoretic fiber of $\phi$ over the closed point corresponding 
to $D$.
\begin{proposition}\label{fibres are of same dimension}
The fiber $\mc Q_{D}$ is irreducible of dimension $d(r-1)$.
\end{proposition}
\begin{proof}
We define a morphism $\prod \mc Q_{d_i[c_i]}\to \mc Q_D$ as follows.
Let $p_{j}$ be the projections $C\times \prod \mc Q_{d_i[c_i]} \to C\times \mc Q_{d_j[c_j]}$
and $p$ be the projection $C\times  \prod \mc Q_{d_i[c_i]} \to C$.
Let $B_{d_i[c_i]}$ denote the universal quotient over $C\times \mc Q_{d_i[c_i]}$.
Then over $C\times \prod \mc Q_{d_i[c_i]}$, we define a quotient
$$p^*E\to \bigoplus p^*_{i}B_{d_i[c_i]}$$
Clearly, $\bigoplus p^*_iB_{d_i[c_i]}$ is flat, and hence induces a morphism 
\begin{equation}\label{theta_D}
\theta_D:\prod \mc Q_{d_i[c_i]}\to \mc Q 
\end{equation}
which is bijective onto the closed points of $\mc Q_D$.
Therefore, $\mc Q_D$ is irreducible.
Since the dimension of the general fiber of $\phi$ is $d(r-1)$, we get 
$$d(r-1)\leq {\rm dim}\,\mc Q_D\leq\sum {\rm dim }\,\mc Q_{d_i[c_i]}=\sum d_i(r-1)=d(r-1)$$
This proves the corollary.
\end{proof}

\begin{corollary}
The map $\phi$ is flat. 
\end{corollary}
\begin{proof}
This follows using \cite[Chapter III, Exercise 10.9]{Ha}
\end{proof}

\begin{corollary}
The fiber $\mc Q_{d[c]}$ is reduced, irreducible and normal.
In particular, it is integral.
\end{corollary}
\begin{proof}
Since $\phi$ is flat and $\mc Q$ is smooth, 
it follows from \cite[\href{https://stacks.math.columbia.edu/tag/045J}{Tag 045J}]{Stk}
(or see Corollary to \cite[Theorem 23.3]{Mat})
that the fiber $\mc Q_{d[c]}$ is Cohen-Macaulay. Thus, the fiber 
satisfies Serre's condition $(S_2)$. Now from corollary \ref{fiber-R1}
it follows that the fiber satisfies $(R_0)$ and $(S_1)$ and so it 
is reduced. Since it satisfies $(R_1)$ and $(S_2)$ it is normal.
\end{proof}

\begin{lemma}
$\mc Q_{D}\cong\prod \mc Q_{d_i[c_i]}$.
\end{lemma}
\begin{proof}
The map $\theta_D$ in \eqref{theta_D} sits in 
a commutative diagram 
\[
 \xymatrix{
	\prod\mc Q_{d_i[c_i]}\ar[r]^{\theta_D}\ar[d] & \mc Q\ar[d]\\
	\prod S^{d_i}C\ar[r] & S^dC
	}
\]
From the above diagram it is clear that $\theta_D$ factors to give a map 
$$\tilde\theta_D:\prod\mc Q_{d_i[c_i]}\to \mc Q_D\,.$$
We want to give a map in the other direction.
Let $p_D:C\times \mc Q_D \to C$ be the first projection.
Let us denote the restriction of the universal quotient to $C\times \mc Q_D$ by
$$p^*_D E \to B_D\,.$$
There are integers $e_i$ such that the quotient $B_D$ is supported 
on the following closed subscheme of $C\times \mc Q_D$
$$\bigsqcup_i\,{\rm Spec}\, (\mc O_C/\mf{m}_{C,c_i}^{e_i})\times \mc Q_D\,.$$
Let $\mf{j}_i:{\rm Spec}\, (\mc O_C/\mf{m}_{C,c_i}^{e_i})\times \mc Q_D\hookrightarrow C\times \mc Q_D$
denote the closed immersion.
Let 
$$B_{d_i[c_i]}:=\mf{j}_{i*}\Big(B_{D}\vert_{{\rm Spec}\, (\mc O_C/\mf{m}_{C,c_i}^{e_i})\times \mc Q_D}\Big)\,.$$
Clearly, since $B_D$ is flat over $\mc Q_D$, the $B_{d_i[c_i]}$ is also 
flat over $\mc Q_D$.
We define the quotients 
$$p^*E\to B_D \to B_{d_i[c_i]}\,,$$
which defines a morphism 
$\mc Q_D\to \mc Q_{d_i[c_i]}$. This defines a morphism 
$\gamma_D:\mc Q_D\to \prod \mc Q_{d_i[c_i]}$. One easily checks 
that the pullback along $id_C\times (\theta_D\circ \gamma_D)$ of the universal quotient 
is $p_1^*E\to B_D$. This shows that $\tilde \theta_D\circ \gamma_D$
is the identity. Arguing as in the last para of the proof of 
proposition \ref{birational}, the lemma is proved.
\end{proof}

\begin{corollary}\label{cor-fiber}
The fiber $\mc Q_D$ is reduced, irreducible and normal.
\end{corollary}

\section{Main Theorem}
\begin{definition}\label{cat-nf}
	Let $X$ be a connected, projective and reduced $k$-scheme. Let $\mathcal{C}^{\rm nf}(X)$ denote the full subcategory of 
	coherent sheaves whose objects are coherent sheaves $E$ on $X$ satisfying the following two conditions: 
	\begin{enumerate}
		\item $E$ is locally free, and 
		\item for any smooth projective curve $C$ over $k$ and any morphism $f : C \longrightarrow X$, 
		the vector bundle $f^*E$ is semistable of degree $0$. 
	\end{enumerate}
\end{definition}
We call the objects of the category $\mc C^{\rm nf}(X)$ {\it numerically flat vector bundles} on $X$. 
Fix a $k$-valued point $x \in X$. 
Let ${\rm Vect}_k$ be the category of finite dimensional $k$-vector spaces. 
Let $T_x : \mc C^{\rm nf}(X) \longrightarrow {\rm Vect}_k$ be the fiber functor defined by 
sending an object $E$ of $\mc C^{\rm nf}(X)$ to its fiber $E_x \in {\rm Vect}_k$ at $x$. 
Then $(\mc C^{\rm nf}(X), \otimes, T_x, \mathcal{O}_X)$ is a neutral Tannaka category 
\cite[Proposition 5.5, p.~2096]{La}. 
The affine $k$-group scheme $\pi^S(X, x)$ Tannaka dual to this category is called the 
{\it S-fundamental group scheme} of $X$ with base point $x$ \cite[Definition 6.1, p.~2097]{La}. 

A vector bundle $E$ is said to be {\it finite} if there are distinct non-zero polynomials 
$f, g \in \mathbb{Z}[t]$ with non-negative coefficients such that $f(E) \cong g(E)$. 

\begin{definition}
	A vector bundle $E$ on $X$ is said to be {\it essentially finite} if there exist two 
	numerically flat vector bundles $V_1, V_2$ and 
	finitely many finite vector bundles $F_1, \ldots, F_n$ on $X$ with 
	$V_2 \subseteq V_1 \subseteq \bigoplus\limits_{i=1}^n F_i$ such that $E \cong V_1/V_2$. 
\end{definition}

Let ${\rm EF}(X)$ be the full subcategory of coherent sheaves 
whose objects are essentially finite vector bundles on $X$. 
Fix a closed point $x \in X$ and let 
$T_x : {\rm EF}(X) \longrightarrow {\rm Vect}_k$
be the fiber functor defined by 
sending an object $E \in {\rm EF}(X)$ to its fiber $E_x$ at $x$. 
Then the quadruple $({\rm EF}(X), \bigotimes, T_x, \mathcal{O}_X)$ is a neutral Tannakian category. 
The affine $k$-group scheme $\pi^N(X, x)$ Tannaka dual to this category is referred to as 
the {\it Nori-fundamental group scheme} of $X$ with base point $x$, see 
\cite{No1} for more details. 

In \cite[Proposition 8.2]{La} it is proved that the $S$-fundamental group 
of projective space is trivial. In \cite{Mehta-Hogadi} it is proved 
that the $S$-fundamental group scheme is a birational invariant of
smooth projective varieties.

Let $S_{d_i,c_i}$ be the space defined in section \ref{S_d} by taking 
$d=d_i$ and $c=c_i$. In view of the discussion in section 
\ref{subsection-birational} there is a birational map 
$$\eta_D:=(\tilde \theta_D\circ \prod g_{d_i}):\prod S_{d_i,c_i}\to \prod \mc Q_{d_i[c_i]}\to \mc Q_D\,.$$

\begin{proposition}\label{prop-num-flat-trivial}
A numerically flat bundle on $\mc Q_D$ is trivial.
\end{proposition}
\begin{proof}
Let $W$ be a numerically flat bundle on $\mc Q_D$.
As $\mc Q_D$ is normal, $\eta_D$ is birational, $\prod S_{d_i,c_i}$ 
is a smooth rational variety, we have 
\begin{align*}
	W &\cong \eta_{D*}\eta_D^*W\\
		&\cong \eta_{D*}(\mc O)^{\oplus r}\\
		&\cong \mc O_{\mc Q_D}^{\oplus r}
\end{align*}
In the above we have used the result of \cite{Mehta-Hogadi}. This proves the proposition. 
\end{proof}

We now prove the main result of this article.

\begin{theorem}\label{m-t}
Let $k$ be an algebraically closed field. 
Let $C$ be an irreducible smooth projective curve over $k$. 
Let $E$ be a locally free sheaf on $C$ of rank $\geq 2$.
Fix an integer $d \geq 2$. Let $\mc Q$ denote the Quot scheme
parameterizing torsion quotients of $E$ of degree $d$. 
Let $S^dC$ denote the $d$th symmetric product of $C$ and 
let $\phi:\mc Q\to S^dC$ denote the Hilbert Chow map (see section \ref{section-Hilbert-Chow}). 
Then the induced map $\phi^S_*:\pi^S(\mc Q,q)\to \pi^S(S^dC,\phi(q))$ is an isomorphism.
\end{theorem}
\begin{proof}
Since the fibers of $\phi$ are projective integral varieties, and $\phi$ is flat, 
it follows that $\phi_*(\mc O_{\mc Q_D})=\mc O_{S^dC}$. Now applying \cite[Lemma 8.1]{La}
we see that $\phi^S_*$ is faithfully flat. To prove $\phi^S_*$ is a closed immersion
we will use \cite[Proposition 2.21(b)]{DMOS}, which we recall for the convenience of the reader. 
For an affine algebraic group scheme $G$ over $k$, let ${\rm Rep}_k(G)$ 
denote the category of finite dimensional representations of $G$ on 
$k$-vector spaces.
Let $\theta : G \to G'$ 
be a homomorphism of affine group schemes over $k$ and let 
\begin{equation}\label{eqn-hom-f}
	\widetilde{\theta} : {\rm Rep}_k(G') \to {\rm Rep}_k(G) 
\end{equation}
be the functor given by sending $\rho' : G' \to {\rm GL}(V)$ to $\rho'\circ \theta : G \to {\rm GL}(V)$. 
An object $\rho : G \to {\rm GL}(V)$ in ${\rm Rep}_k(G)$ is said to be a {\it subquotient} of an object 
$\eta : G \to {\rm GL}(W)$ in ${\rm Rep}_k(G)$ if there are two $G$-submodules $V_1 \subset V_2$ of $W$ 
such that $V \cong V_2/V_1$ as $G$-modules. 	
Let $\theta : G \to G'$ be a homomorphism of 
affine algebraic groups over $k$. Then $\theta$ is a closed 
immersion if and only if every object of ${\rm Rep}_k(G)$ 
is isomorphic to a subquotient of an object of the form $\widetilde{\theta}(V')$, 
for some $V' \in {\rm Rep}_k(G')$. 

Let $W$ be a numerically flat bundle on $\mc Q$ which corresponds
to a finite dimensional representation $\rho:\pi^S(\mc Q,q)\to {\rm GL}(V)$.
We will show that there is a numerically flat bundle $W'$ on $S^d(C)$ 
such that $W\cong \phi^*W'$. This precisely means that there is a representation
$\rho':\pi^S(S^d(C),\phi(q))\to {\rm GL}(V)$ such that $\rho=\rho'\circ \phi_*^S$.
Now by the criterion in the preceding para we see that $\phi_*^S$
is a closed immersion. 

By Grauert's theorem 
\cite[Corollary 12.9]{Ha} and Proposition \ref{prop-num-flat-trivial}, it follows that 
if $W$ is a numerically flat bundle on $\mc Q$ then $\phi_*(W)$ is a locally 
free sheaf on $S^dC$ and the natural map $\phi^*\phi_*(W)\to W$ is an 
isomorphism. It follows easily that $\phi_*(W)$ is numerically flat. 
This is easily checked because given a morphism $f:X\to Y$ between
two projective varieties, and a morphism from a projective curve 
$C\to Y$, we can always find a cover $C'\to C$ such that the composite 
$C'\to Y$ lifts to a map $C'\to X$. 
This proves that $\phi^S_*$ is a closed immersion.
\end{proof}

From the $S$-fundamental group scheme we recover the Nori fundamental group scheme
as the inverse limit of finite quotients. Similarly, the etale 
fundamental group scheme can be recovered as the inverse limit of finite and 
reduced quotients. Thus, we get the following corollary. (See \S 5.5 in \cite{PS-surface}
for more details.)

\begin{corollary}\label{m-c}
The induced map $\phi^N_*:\pi^N(\mc Q,q)\to \pi^N(S^dC,\phi(q))$ is an isomorphism.
The induced map $\phi^{\et}_*:\pi^{\et}(\mc Q,q)\to \pi^{\et}(S^dC,\phi(q))$ is an isomorphism.
\end{corollary}



\newcommand{\etalchar}[1]{$^{#1}$}

\end{document}